\documentclass[11pt]{amsart}
\usepackage{amsfonts,amsmath,amssymb,fullpage,graphicx, comment}
\usepackage{verbatim}
\usepackage{xcolor}
\usepackage{hyperref}

\usepackage{array}
\usepackage{booktabs}

\usepackage{mathrsfs}

\newtheorem{thm}{Theorem}[section]

\newtheorem{lem}[thm]{Lemma}

\newtheorem{prop}[thm]{Proposition}
\theoremstyle{definition}

\theoremstyle{remark}
\newtheorem{rem}[thm]{\bf{Remark}}
\numberwithin{equation}{section}

\newcommand{\beas}{\begin{eqnarray*}}
\newcommand{\eeas}{\end{eqnarray*}}
\newcommand{\bes} {\begin{equation*}}
\newcommand{\ees} {\end{equation*}}
\newcommand{\be} {\begin{equation}}
\newcommand{\ee} {\end{equation}}
\newcommand{\bea} {\begin{eqnarray}}
\newcommand{\eea} {\end{eqnarray}}

\title{Extreme Values of Quadratic Dirichlet $L$-functions over Monic Irreducible Polynomials in $\mathbb{F}_q[t]$}

\begin{document}

\author{Ahammad Mostafa Hossain}
	\address{Ahammad Mostafa Hossain\\ Department of Mathematics \\
		Indian Institute of Technology Kharagpur \\
			Kharagpur-721302,  India.} 
	\email{mostafa.24@kgpian.iitkgp.ac.in}


\begin{abstract}
	In this paper, we establish a new $\Omega$-result for the values
	$|L(\sigma,\chi_P)|$ of quadratic Dirichlet $L$-functions at and near
	the central point, where
	\[
	\sigma=\frac12+\frac{\kappa}{\ln g}, \qquad \kappa\geq 0,
	\]
	and $P$ ranges over monic irreducible polynomials associated with
	hyperelliptic curves of genus $g$ over a fixed finite field
	$\mathbb{F}_q$. We consider the asymptotic setting in which $q$ is fixed
	and $g\to\infty$. More precisely, for every $\epsilon \in (0,1/2)$ and
	every fixed $\kappa\geq 0$, we prove that
	\[
	\max_{P \in \mathcal{P}_{2g+1}}
	|L(\sigma,\chi_P)|
	\gg
	\exp \left(
	\left(
	e^{-\kappa} \ln q \,
	\sqrt{
		\frac{\sqrt{q}+1}{\sqrt{q}-1}
		(1/2-\epsilon)
	}
	+o(1)
	\right)
	\sqrt{\frac{g\ln_2 g}{\ln g}}
	\right),
	\]
	where $\mathcal{P}_{2g+1}$ is the set of all monic irreducible
	polynomials of degree $2g+1$ in $\mathbb{F}_q[t]$. In particular, when
	$\kappa=0$, this gives the corresponding result at the central point
	$\sigma=1/2$. Our result extends the recent work of Darbar and Maiti
	(2024) and yields an improved lower bound for the extreme values in this
	family. In addition, for $1/2<\sigma<1$ and sufficiently large $n$, we
	study the extreme values of $L(\sigma,\chi_P)$, where
	$P\in\mathcal{P}_n$, and obtain an improved lower bound compared with
	the result of Lumley (2021).
\end{abstract}

\maketitle
\section{Introduction}

A fundamental challenge in the theory of $L$-functions is to establish sharp lower bounds for their extreme values. Considerable progress has been achieved in this direction in recent years, with important advances concerning the Riemann zeta function and families of Dirichlet $L$-functions in the conductor aspect.

We adopt the convention that, for each integer $j \geq 2$, the symbol $\ln_j$ represents the $j$-fold iterated natural logarithm, while $\log_j$ denotes the corresponding $j$-fold iterated logarithm with base $q$, where $q \geq 3$ is fixed. 
Soundararajan \cite{soundararajan2008} developed the resonance method to study
extreme values of several families of $L$-functions, including the Riemann zeta
function and quadratic Dirichlet $L$-functions. This method has since become an
important tool in the study of large central values of $L$-functions.
De la Bret\`eche and Tenenbaum \cite{breteche2019} further developed this
approach and proved that
\begin{equation}\label{tene re 2}
\max_{\substack{\chi \in X_q^+ \\ \chi \neq \chi_0}} |L(1/2, \chi)| 
\ge \exp \left( (1+o(1)) \sqrt{\frac{\ln q \, \ln_3 q}{\ln_2 q}} \right),
\end{equation}
where $X_q^+$ denotes the family of even Dirichlet characters modulo $q$.
In the present work, we use this method to investigate large central values of
quadratic Dirichlet $L$-functions over function fields.

Over the past ten years, alongside the progress in number fields, there has been growing interest in understanding the behavior of such functions in the function field setting. In a recent work, Doki\'c et al. \cite{dokic2020} established analogues of \eqref{tene re 2} and the main results from \cite{aist2019} for rational function fields over finite fields. More specifically, they studied the occurrence of large values of $L(\sigma,\chi)$ in the function field setting, allowing $\chi$ to range over Dirichlet characters modulo a polynomial $P \in \mathbb{F}_q[t]$, with $\sigma \in [1/2,1]$. When $P$ is a monic irreducible polynomial with $|P|=q^{\mathbf{d}(P)}$, 
their argument based on estimates for the full GCD sum gives the following asymptotic lower bound as $\mathbf{d}(P)\to\infty$:
\[
\max_{\substack{\chi\pmod P \\ \chi \text{ even} \\ \chi \neq \chi_0}} 
|L(1/2, \chi)| 
\ge \exp \left( \left( \frac{1}{2} \sqrt{\frac{\sqrt{q} + 1}{\sqrt{q} - 1} \ln q} + o(1) \right) 
\sqrt{\frac{\ln |P| \, \ln_3 |P|}{\ln_2 |P|}} \right).
\]

\subsection{Extreme  values of Dirichlet \texorpdfstring{$L$}{L}-functions over irreducible polynomials.}

The central goal of this paper is to investigate the function field analogue of 
the quadratic Dirichlet $L$-functions. We 
begin by recalling some essential notation from the theory of function fields. Fix an odd prime power $q$, and write $\mathbb{F}_q$ for the finite field
with $q$ elements. Throughout, we work in the polynomial ring
$\mathbb{F}_q[t]$. For each monic irreducible polynomial
$P\in\mathbb{F}_q[t]$, let $\chi_P(A):=\left(\frac{A}{P}\right) $  denote the quadratic character modulo $P$, defined through the quadratic
residue symbol in $\mathbb{F}_q[t]$.
 For
$A \in \mathbb{F}_q[t]$, the quadratic residue symbol modulo $P$ is defined by
\[
\left(\frac{A}{P}\right) =
\begin{cases}
	1, & \text{if } A \text{ is a square}\pmod{P}, \quad P \nmid  A,\\
	-1, & \text{if }  A \text{ is not a square}\pmod{P}, \quad P \nmid  A,\\
	0, & \text{if } P \mid  A.
\end{cases}
\]
We then define the Dirichlet $L$-function attached to $\chi_P$ by
\[
L(s,\chi_P) = \sum_{A\in \mathcal{M}} \frac{\chi_P(A)}{|A|^s}
, \qquad \Re(s) > 1.
\] 
 Let $\mathcal{P}_n$ denote the set of all monic irreducible polynomials
of degree $n$ in $\mathbb{F}_q[t]$. For each $P\in\mathcal{P}_n$, we consider the hyperelliptic curve given in affine form by
\[
C_P:y^2=P(t).
\]
The corresponding curve is nonsingular, and its genus $g$ is determined by
\[
2g=n-\lambda -1,
\]
where
\[
\lambda=
\begin{cases}
	0, & \text{if } n \text{ is odd},\\
	1, & \text{if } n \text{ is even}.
\end{cases}
\]

Darbar and Maiti \cite{darbar2024} investigated the behavior of the central values
$L(1/2,\chi_P)$ as $P$ ranges over $\mathcal{P}_{2g+1}$, where $g$ denotes
the genus of the corresponding hyperelliptic curve. They proved
that, for every $\epsilon\in(0,1/2)$,
\[
\max_{P \in \mathcal{P}_{2g+1}} |L(1/2, \chi_P)| 
\gg \exp \left( \left( \sqrt{(1/2 - \epsilon) \ln q} + o(1) \right) 
\sqrt{\frac{g \ln_2 g}{\ln g}} \right),
\]
as $g \to \infty$.

We improve the above result of Darbar and Maiti, and obtain the following theorem.

\begin{thm}\label{main thm}
   Let $q$ be a fixed odd prime power and $\kappa \geq 0$ be any fixed real number.  Set 
   \[
   \sigma= \frac12+\frac{\kappa}{\ln g}.
   \]
   For any $\epsilon\in (0,1/2)$,  we obtain, as $g \to \infty$, 
    \[
    \max_{P \in \mathcal{P}_{2g+1}} |L(\sigma, \chi_P)| 
    \gg \exp \left( \left(  e^{-\kappa} \,\ln q \,  \sqrt{\frac{\sqrt{q}+1}{\sqrt{q}-1} (1/2-\epsilon)}  + o(1) \right) 
\sqrt{\frac{g \ln_2 g}{\ln g}} \right).
    \]
\end{thm}

\begin{rem}
	For the rational function field $F=\mathbb F_q[t]$, our leading
	constant is larger than the corresponding constant in
	\cite[Theorem~1.3]{dong2026}, provided that $\varepsilon>0$ is chosen
	sufficiently small. We note, however, that the theorem of Dong and Liu
	applies to arbitrary fixed global function fields and also shows that
	many principal primes $P$ give rise to quadratic characters $\chi_P$
	with large $L$-values.
\end{rem}

\begin{rem}
	For $P\in\mathcal{P}_{2g+2}$, it seems difficult to obtain results of
	the same strength as Theorem~\ref{main thm}  by the present method.
	Indeed, in the even-degree case we have $\lambda=1$, and the
	approximate functional equation contains additional $\lambda$-terms (see \cite[Lemma 7]{darbar2024}).
	These extra terms do not preserve the positivity used in our
	resonator argument. Therefore, our present approach does not directly
	give the same lower bounds for the family
	$\{L(\sigma,\chi_P):P\in\mathcal{P}_{2g+2}\}$.
\end{rem}

We next investigate the extreme values of 
 $L(\sigma,\chi_P)$ inside the critical strip for $P \in \mathcal{P}_{n}$.
 Lumley \cite[Theorem~1.3]{lumley2021} proved that, for $q\equiv 1 \pmod{4}$, fixed $1/2<\sigma<1$, and sufficiently large $\mathbf{N}$, there exists a monic irreducible polynomial $P$ of degree $\mathbf{N}$ such that 
\begin{equation}\label{lumley rel}
	\ln L(\sigma,\chi_P) \geq  \beta'_q(\sigma)\frac{ ( \log |P|)^{1-\sigma}}{(\log \log |P|)^{\sigma}},
\end{equation}
where
\[
\beta'_q(\sigma) = \frac{\mathcal{Z}(2-\sigma)}{(10 q \mathcal{Z}(2))^{1-\sigma}}.
\]
Here, $|P|$ and $\mathcal{Z}(s)$ are as defined in Section~2. 

Using the method of Darbar and Maiti~\cite{darbar2025large} for studying
large values of quadratic Dirichlet $L$-functions in the number field
setting, where they employed the long resonator technique to obtain lower bound
for $\log L(\sigma,\chi_d)$ with $d$ ranging over fundamental discriminants, we establish the following lower bound for the corresponding family over function fields. Before presenting our result, we
first consider 
\begin{equation}\label{assumtion}
	\beta_q(\sigma,b)= b \mathcal{Z}(2-\sigma) (\beta_q(b))^{\sigma-1}, \quad \beta_q(b) = q \alpha_q(b),\quad\text{and } \quad \alpha_q(b)
	=
	\frac{2\mathcal{Z}(2)}{\ln q}
	\ln\left(\frac{(1+b)^2}{1+b^2}\right)
\end{equation}
for every fixed $\sigma \in (1/2,1)$, and $b \in (0,1)$.

\begin{thm}\label{thm inside critical strip}
	Let $q$ be a fixed odd prime power and $\beta_q(\sigma,b), \alpha_q(b)$ be as defined in \eqref{assumtion}. Then, for every sufficiently large  $\mathbf{N}$, there exists a polynomial  $P \in \mathcal{P}_\mathbf{N}$ such that
	\begin{equation*}\label{lower bound inside strip}
		\ln L(\sigma,\chi_P) \geq (\beta_q(\sigma,b)+o(1))\frac{ ( \log |P|)^{1-\sigma}}{(\log \log |P|)^{\sigma}}.
	\end{equation*}
	 Furthermore, for every  $	0<\xi<\frac{1}{\alpha_q(b)},$
	define
	$
		\tau_{\xi,\mathbf N}
		:=
		\beta_q(\sigma,b)
		\left(1-\xi\alpha_q(b)\right)^{1-\sigma}
		{\mathbf N^{1-\sigma}}
		{(\log\mathbf N)^{ -\sigma}},
	$ such that 
	\begin{equation*}
		\Phi_{\mathbf N}(\sigma,\xi)
		:=
		\frac{1}{	\# \mathcal P_{\mathbf N} }
		\#\left\{
		P\in\mathcal P_{\mathbf N}:
		\ln L(\sigma,\chi_P)
		\geq \tau_{\xi,\mathbf N}
		\right\}.
	\end{equation*}
		Then, for sufficiently large $\mathbf N$,
	\begin{equation*}
		\Phi_{\mathbf N}(\sigma,\xi)
		\geq
		|P|^{-\frac{1}{2}
			(1-\xi\alpha_q(b))+o(1)}.
	\end{equation*}
	
\end{thm}

\begin{rem}
	Letting $ b \to 1$ in the above theorem, we obtain
	\[
	\beta_q(\sigma,b)=\mathcal{Z}(2-\sigma) \left( \frac{\ln q}{2 \mathcal{Z}(2)q \ln 2} \right) ^{1-\sigma}.
	\]
	This improves the constant 
	$
	\beta'_q(\sigma) = {\mathcal{Z}(2-\sigma)}{(10 q \mathcal{Z}(2))^{\sigma-1}},
	$
	obtained by Lumley \cite{lumley2021}. Thus, our result gives the same order of magnitude as Lumley's
	$\Omega$-result \eqref{lumley rel}, with a strictly larger leading constant.
\end{rem}

The endpoint $\sigma=1$ has also been studied previously. In particular,
it was shown in \cite[Theorem 1.6]{lumley2019complex} that, for sufficiently large $\mathbf{N}$, there exists
an irreducible polynomial $Q_1$ of degree  $\mathbf{N}$ such that
\[
L(1,\chi_{Q_1})
>
e^{\gamma}
\left(
\log_2 |Q_1|+\log_3 |Q_1|
\right)
+O(1),
\]
where $\gamma$ denotes the Euler--Mascheroni constant.

The structure of this article is as follows. Section~\ref{section preli} is devoted to the relevant background
and preliminary results used in the subsequent proofs. In Section~\ref{Gal sum sec}, we obtain a lower bound
for the G\'al sum associated with a set of monic squarefree polynomials, which is then applied
in the proof of Theorem~\ref{main thm}. Sections~\ref{proof thm 1} and~\ref{section proof inside critical strip} contain the proofs of Theorems~\ref{main thm} and~\ref{thm inside critical strip},
respectively.

\section{Preliminaries}\label{section preli}

Throughout, let $q=p^r$, where $p$ is an odd prime and $r\geq 1$, and write $\mathcal{A}:=\mathbb{F}_q[t]$ for the polynomial ring over the finite field $\mathbb{F}_q$. For $A\in\mathcal{A}$, we use $\mathbf{d}(A)$ to denote its degree. If
$A\neq 0$, its norm is given by $ |A|:=q^{\mathbf{d}(A)},$ while, by convention, $|0|:=0$. We shall refer to a monic irreducible element of $\mathcal{A}$ as a prime polynomial.

We write $\mathcal{M}$ for the collection of monic polynomials in $\mathcal{A}$, while $\mathcal{P}$ denotes the collection of monic
irreducible polynomials.  For $n\in\mathbb{N}$, we define the following subcollections of $\mathcal{M}$ and $\mathcal{P}$:
\begin{align*}
    \mathcal{M}_n &= \{ A \in \mathcal{M} : \mathbf{d}(A) = n \},\\
    \mathcal{M}_{\leq n} &= \{ A \in \mathcal{M} : \mathbf{d}(A) \leq n \},\\
    \mathcal{P}_n &= \{ P \in \mathcal{P} : \mathbf{d}(P) = n \},\\
    \mathcal{P}_{\leq n} &= \{ P \in \mathcal{P} : \mathbf{d}(P) \leq n \},
\end{align*}
 and let
\[
\pi_q(n):=\#\mathcal{P}_n=|\mathcal{P}_n|
\qquad\text{and}\qquad
\Pi_q(n):=\#\mathcal{P}_{\leq n}=|\mathcal{P}_{\leq n}|.
\]
The zeta function associated with $\mathcal{A}$ is given by the infinite series
\[
\mathcal{Z}(s) = \sum_{A \in \mathcal{M}} \frac{1}{|A|^s}
= \prod_{P \in \mathcal{P}} \left(1 - \frac{1}{|P|^{s}}\right)^{-1}, \quad \Re(s)>1.
\]
Since there are exactly $q^n$ different monic polynomials of degree $n$, the zeta
function admits the explicit expression
\[
\mathcal{Z}(s)
=\sum_{n\geq 0} q^{n(1-s)}
=\frac{1}{1-q^{1-s}}.
\]
Thus, $\mathcal{Z}(s)$ extends meromorphically to the entire complex plane,
with its only pole being a simple pole at $s=1$. 

The quadratic Dirichlet $L$-function associated with $\chi_P$ admits the Euler product
\[
L(s,\chi_P)
=
\prod_{Q\in\mathcal{P}}
\left(1-\frac{\chi_P(Q)}{|Q|^s}\right)^{-1},
\qquad \Re(s)>1.
\]
After making the change of variable $u=q^{-s}$, the above $L$-function will also be viewed in the $u$-variable and denoted by $\mathcal{L}(u,\chi_P)$, where
\[
\mathcal{L}(u,\chi_P) =\sum_{A \in \mathcal{M}} \frac{\chi_P(A)}{|A|^s} =\prod_{Q \in \mathcal{P}} \left(1-\chi_P(Q) u^{\mathbf{d}(Q)}\right)^{-1},
\qquad |u|<\frac{1}{q}. 
\]
We refer to
\cite{rosen2002} for the relevant background.

For $A \in \mathcal{A}$, the function field analogue of the von Mangoldt function $\Lambda_{\mathcal{A}}(A)$ is defined as:
\[
\Lambda_{\mathcal{A}}(A) =
\begin{cases} 
	\mathbf{d}(P), \quad \text{if} \quad A=P^{k} \quad \text{for some monic irreducible polynomial and } k\geq 1,
	\\
	0,\quad \quad \quad \text{otherwise.}
\end{cases}
\]

We now present the approximate functional equation for $L(s,\chi_P)$ in the case $P \in \mathcal{P}_{2g+1}$.

\begin{lem}[Approximate Functional Equation]
	\label{ApproximateFE}
Suppose that $P\in\mathcal{P}_{2g+1}$ and let $\chi_P$ denote the corresponding quadratic Dirichlet character. Then
\[
L\left(s,\chi_P\right)
=
\sum_{A\in\mathcal{M}_{\leq g}}
\frac{\chi_P(A)}{|A|^s}
+ q^{g(1-2s)}
\sum_{A\in\mathcal{M}_{\leq g-1}}
\frac{\chi_P(A)}{|A|^{1-s}}.
\]
\end{lem}

\begin{proof}
 See	\cite[Eq.~(3.23)]{andra2012}.
\end{proof}

The following lemmas, which are essential for our purposes, provide a bound for 
character sums over irreducible polynomials and is a consequence of the Riemann 
hypothesis for curves.

\begin{lem}[Weil bound]
\label{weil}
For any non-square polynomial $A\in\mathcal{M}$ and any positive integer $n$,
\[
\left|
\sum_{P\in\mathcal{P}_n}\chi_P(A)
\right|
\ll
\frac{q^{n/2}}{n}\mathbf{d}(A).
\]
\end{lem}
\begin{proof}
See \cite[Eq.~(2.5)]{rudnick2010}.
\end{proof}

\begin{lem}[Prime Polynomial Theorem]
	\label{prime pol thm}
	Assume that $n$ is a large positive integer. Then
\[
\pi_q(n)= \frac{q^n}{n} + O\left(\frac{q^{n/2}}{n}\right), \quad\text{and} \quad
\pi_q(n) \leq \frac{q^n}{n}.
\]
\end{lem}
\begin{proof}
	See \cite[Theorem~2.2]{rosen2002}.
\end{proof}

\begin{lem}
	\label{char sum}
Let $A \in \mathcal{M}$. Then
\[
\sum_{P \in \mathcal{P}_n} \chi_P(A^2) = \frac{q^n}{n} + O\left(\frac{q^{n/2}}{n}\right) + O(\log \log \mathbf{d}(A)).
\]
\end{lem}

\begin{proof}
See \cite[Lemma~5]{darbar2024}.
\end{proof}

\begin{lem}\label{produc of poly}
Assume that $0\le s<1$. Then, for a large positive integer $n$,
\[
\sum_{P \in \mathcal{P}_{\leq n}} \frac{1}{|P|^s}
= \mathcal{Z}(2-s)\frac{q^{n(1-s)}}{n}\left(1 + O\left(\frac{\ln n}{n}\right)\right).
\]
For the special case $s=0$, one obtains
\[
\Pi_q(n) = \mathcal{Z}(2)\frac{q^n}{n}\left(1 + O\left(\frac{\ln n}{n}\right)\right).
\]
\end{lem}
\begin{proof}
See	\cite[Lemma~3.2]{dokic2020}.
\end{proof}

\section{Gál's sum over \texorpdfstring{$\mathbb{F}_q[t]$}{F\_q[t]} }\label{Gal sum sec}

Suppose that $\mathfrak{M}$ is a finite family of monic polynomials in $\mathcal{A}$. For $A,B\in\mathcal{M}$, we write $(A,B)$ for their greatest common divisor
and $[A,B]$ for their least common multiple. Recall that the G\'al sum corresponding to $\mathfrak{M}$ is given by
\[
S(\mathfrak{M}) := \sum_{A,B \in \mathfrak{M}} \sqrt{\frac{|(A,B)|}{|[A,B]|}}.
\]
 For an integer $\mathcal{N}$, we define
\[
\Gamma(\mathcal{N}) := \sup_{{|\mathfrak{M}| = \mathcal{N}}} \frac{S(\mathfrak{M})}{|\mathfrak{M}|},
\]
where the supremum is taken as $\mathfrak{M}$ ranges over all finite sets with $\mathcal{N}$ elements.
Assume that
  \[
\sigma= \frac12+\frac{\kappa}{\ln_2 \mathcal{N}},
\]
where $\kappa\geq0$ is a  fixed real number. Define
\[
S_{\sigma}(\mathfrak{M}) := \sum_{A,B \in \mathfrak{M}} \left ( \frac{|(A,B)|}{|[A,B]|} \right)^{\sigma}.
\]
Then clearly $S_{1/2}(\mathfrak{M})=S(\mathfrak{M})$.

Motivated by the ideas of de la Bret\`eche and Tenenbaum \cite{breteche2019}, particularly by their use of G\'al-type sums in deriving lower bounds for Dirichlet $L$-functions, Doki\'c et al.~\cite{dokic2020} developed an analogous framework over the rational function field $\mathbb{F}_q[t]$. In particular, they established a lower bound for $S(\mathfrak{M})$ when $\mathfrak{M}$ is a finite set of monic polynomials and applied this estimate to the study of large values of Dirichlet $L$-functions over function fields. In the present work, we require a corresponding lower bound for $S_{\sigma}(\mathfrak{M})$ under the additional restriction that $\mathfrak{M}$ consists of monic squarefree polynomials. The following result on GCD sums provides the main ingredient needed in the proof of Theorem \ref{main thm}.

\begin{prop}\label{prop}
	Let $q$ be a fixed odd prime power and let
	\[
	\sigma=\frac12+\frac{\kappa}{\ln_2 \mathcal{N}},
	\qquad \kappa\geq 0.
	\]
	Then, uniformly for $\kappa$ in any fixed compact subset of
	$[0,\infty)$, for all sufficiently large $\mathcal{N}$, there exists
	a set $\mathfrak{R}$ of monic squarefree polynomials with
	$|\mathfrak{R}|=\mathcal{N}$ such that
	\[
	S_{\sigma}(\mathfrak{R})
	\geq
	\mathcal{N}
	\exp\left\{
	\left(
	e^{-\kappa}
	\sqrt{\frac{\sqrt{q}+1}{\sqrt{q}-1}\ln q}
	+o(1)
	\right)
	\sqrt{
		\frac{\ln \mathcal{N}\ln_3 \mathcal{N}}
		{\ln_2 \mathcal{N}}
	}
	\right\}.
	\]
\end{prop}
The proof of the above proposition follows the general strategy of Proposition~8.1 in ~\cite{dokic2020}. However, in our setting the set $\mathfrak{R}$ is required to consist entirely of monic squarefree polynomials, and hence the construction of the set $\mathfrak{R}$ has to be modified accordingly. For clarity, we divide the proof into two separate parts. In the first subsection, we construct a set $\mathfrak{R}$ of monic squarefree polynomials with the required cardinality properties. In the second subsection, we establish the desired lower bound for the corresponding Gál sum $S_{\sigma}(\mathfrak{R})$.

\subsection{Construction of the set \texorpdfstring{$\mathfrak{R}$}{R}}
This paragraph is devoted to the construction of sets $\mathfrak{R}$ for which the sum $S_{\sigma}(\mathfrak{R})$ attains large values.  Let $\mathcal{N}$ be a sufficiently large positive integer.
Following the general strategy of
Doki\'c et al.~\cite[Section~8]{dokic2020}, we introduce three
parameters
\[
1<u\leq e, \qquad c>1, \qquad 0<\gamma<1,
\]
which will be specified later.
 For
\[
1\leq k\leq (\ln_2\mathcal{N})^{\gamma},
\]
define integers
\[
i_k=
\left[
k\log u+\log\ln\mathcal{N}+\log\ln_2\mathcal{N}
\right].
\]
Introduce the intervals in $\mathcal{M}$
\[
\mathcal{I}_k = \{Q\in \mathcal{M} :  i_k< \mathbf{d}(Q)\leq i_{k + 1}\}.
\]
Let $P_k$ denote the number of monic irreducible polynomials contained in
$\mathcal{I}_k$.  Thus,
   \[
    P_k =\Pi_q(i_{k+1})-\Pi_q(i_k).
   \]
By Lemma \ref{produc of poly},
\be\label{Pk}
   P_k\leq\mathcal{Z}(2)u^{k+1}\ln\mathcal{N}.
\ee
Moreover, from \cite[p.~1104]{dokic2020}, we have 
   \[
   \Pi_q(i_k)= \mathcal{Z}(2)\ln q\,u^k\ln\mathcal{N}
\left(1+O\left(\frac{k+\ln_3\mathcal{N}}{\ln_2\mathcal{N}}\right)\right),
 \]
uniformly for
$1\leq k\leq(\ln_2\mathcal{N})^\gamma$.
Consequently,
\begin{equation}\label{Pk=}
   P_k=\mathcal{Z}(2)\ln q\,u^k(u-1)\ln\mathcal{N}
\left(1+O\left(\frac{k+\ln_3\mathcal{N}}{\ln_2\mathcal{N}}\right)\right).
\end{equation}
We also put  
\be\label{Jk=}
J_k:=\left[\frac{c\ln\mathcal{N}}{k^2\ln_3\mathcal{N}}\right]
\qquad (1\leq k\leq(\ln_2\mathcal{N})^\gamma),
\ee
and choose $c>1$  such that 
\begin{equation}\label{condition}
c\gamma\ln u<1.
\end{equation}
For each $k$, define the squarefree polynomial
\[
N_k= \prod_{P\in\mathcal{I}_k\cap\mathcal{P}}P.
\]
We consider the set
\[
\mathfrak{R}_k=\left\{R \in \mathcal{M}:R\mid N_k,\ \omega(R)\leq J_k \right\},
\]
where $\omega(R)$ denotes the number of distinct monic irreducible divisors of
$R$. Since $N_k$ is squarefree, every element of $\mathfrak{R}_k$ is also squarefree.

Finally, define
\begin{equation}\label{set R}
\mathfrak{R}=\left\{R=\prod_{1\leq k\leq(\ln_2\mathcal{N})^\gamma}R_k: R_k\in\mathfrak{R}_k \quad (1\leq k\leq(\ln_2\mathcal{N})^\gamma)
\right\}.
\end{equation}
The sets of irreducible factors arising from distinct intervals
$\mathcal{I}_k$ are disjoint. Hence every polynomial in $\mathfrak{R}$ is monic
and squarefree.

\begin{lem}\label{card R}
Under the condition \eqref{condition}, we have
$ |\mathfrak{R}|\leq \mathcal{N}$.
\end{lem}

We omit the proof of this lemma, since it follows closely the argument used in Lemma~2.2 of \cite{breteche2019}. Nevertheless, we record below several estimates that arise from the same reasoning and will be useful in what follows.  
\be\label{|R|}
|\mathfrak{R}|=\prod_{1\leq k\leq(\ln_2\mathcal{N})^\gamma}
|\mathfrak{R}_k|,
\ee
where 
\[
|\mathfrak{R}_k|=\sum_{0\leq j\leq J_k}\binom{P_k}{j}.
\]
It follows readily that 
\[
|\mathfrak{R}_k|\leq 2\binom{P_k}{J_k}\leq 2 \left(  \frac{eP_k}{J_k}\right)^{J_k} \leq \left( e \mathcal{Z}(2)k^2 u^{k+1} \ln_3\mathcal{N} \right)^{J_k}.
\]

\subsection{Lower bound for \texorpdfstring{$S_{\sigma}(\mathfrak{R})$}{R}}
 We have 
 \be\label{SR}
S_{\sigma}(\mathfrak{R}) =\prod_{1\leq k \leq (\ln_2 \mathcal{N})^{\gamma}} S_{\sigma}(\mathfrak{R}_k).
\ee
Let $A$ and $B$  be two elements of $ \mathfrak{R}_k$. Then 
\[
S_{\sigma}(\mathfrak{R}_k)=\sum_{A,B \in \mathfrak{R}_k} \left ( \frac{|(A,B)|}{|[A,B]|}\right)^{\sigma}=\sum_{A,B \in \mathfrak{R}_k} \frac{|(A,B)|^{2\sigma}}{|AB|^{\sigma}}.
\]
For a squarefree polynomial $D$, define
 \[
J_{2\sigma}(D)
:=
\prod_{P\mid D}
\left(|P|^{2\sigma}-1\right).
\]
Then, for every squarefree polynomial $G$, we have
\begin{align*}
	\sum_{D\mid G}J_{2\sigma}(D)
	=
	|G|^{2\sigma}.
\end{align*}
Thus, we have  
\begin{align}
S_{\sigma}(\mathfrak{R}_k) & = \sum_{\substack{D |N_k \\ \omega (D)\leq J_k}} \frac{J_{2\sigma}(D)}{|D|^{2\sigma}} \sum_{\substack{ E_1,E_2 | N_k \\  \omega(E_1),\omega(E_2) \leq J_k-\omega(D) \\ (E_1,D)=(E_2,D)=1 }} \frac{1}{{|E_1E_2|}^{\sigma}} \nonumber \\ 
& = \sum_{\substack{D |N_k \\ \omega (D)\leq J_k}} \frac{J_{2\sigma}(D)}{|D|^{2\sigma}} \left(  \sum_{\substack{ E | N_k \\  \omega(E) \leq J_k-\omega(D) \\ (E,D)=1 }} \frac{1}{{|E|}^{\sigma}} \right)^2.
\end{align}
Since  for all $D | N_k$, we have 
\[
\frac{J_{2\sigma}(D)}{|D|^{2\sigma}} \gg \exp \left( - \sum_{P \in \mathcal{I}_k} \frac{1}{|P|^{2\sigma}} \right),
\]
and 
\begin{align*}
	\sum_{P\in I_k}|P|^{-2\sigma}
	&\ll
	u^{k(1-2\sigma)}
	(\ln\mathcal{N})^{1-2\sigma}
	(\ln_2\mathcal{N})^{-2\sigma}\\
	&\ll \frac1{\ln_2\mathcal{N}}=o(1),
\end{align*}
it follows that 
\be\label{lower SRK}
S_{\sigma}(\mathfrak{R}_k)\gg  \sum_{\substack{D |N_k \\ \omega (D)\leq J_k}} \left(  \sum_{\substack{ E | N_k \\  \omega(E) \leq J_k-\omega(D) \\ (E,D)=1 }} \frac{1}{{|E|}^{\sigma}} \right)^2.
\ee
Let 
\[
j_k:=\left[  \frac{\alpha}{k}  \sqrt{\frac{\ln \mathcal{N}}{\ln_2 \mathcal{N}\ln_3 \mathcal{N}}} \right],
\]
where $\alpha$ is a parameter whose value will be chosen optimally. 
We restrict the summation over $D$ to those polynomials satisfying $ \omega(D)\leq J_k-j_k,$  while the inner summation over $E$ is restricted to polynomials satisfying  $ \omega(E)= j_k.$

Then, we have  
\be
\sum_{\substack{E\mid N_k\\
		\omega(E)\leq J_k-\omega(D)\\
		(E,D)=1}}
\frac{1}{|E|^{\sigma}}
\geq
\sum_{\substack{E\mid N_k\\
		\omega(E)=j_k\\
		(E,D)=1}}
\frac{1}{|E|^{\sigma}}
\gg
\frac{1}{j_k!}
\left(
\sum_{\substack{P\in \mathcal{I}_k\\
		P\nmid D}}
\frac{1}{|P|^{\sigma}}
\right)^{j_k}.\label{>>1/j_k!}
\ee
Using Lemma~\ref{produc of poly} (uniformly for $\sigma$ in a fixed compact neighbourhood of $1/2$),  we obtain
\begin{align}
	\sum_{P \in_{\le i_k}}\frac1{|P|^\sigma}
	={}&\mathcal{Z}(2-\sigma)\ln q \,
	u^{k(1-\sigma)}  
	(\ln\mathcal{N})^{1-\sigma}
	(\ln_2\mathcal{N})^{-\sigma}
	\notag\\
	&\times\left(1+O\!\left(
	\frac{k+\ln_3\mathcal{N}}{\ln_2\mathcal{N}}
	\right)\right).
	\label{eq:cumulative}
\end{align}
Therefore, subtracting \eqref{eq:cumulative} at $i_{k+1}$ and $i_k$ gives
\begin{align}
	\sum_{P\in \mathcal{I}_k}\frac1{|P|^\sigma}
	={}&\mathcal{Z}(2-\sigma)\ln q\,
	u^{k(1-\sigma)}(u^{1-\sigma}-1)
	(\ln\mathcal{N})^{1-\sigma}
	(\ln_2\mathcal{N})^{-\sigma}
	\notag\\
	&\times\left(1+O\!\left(
	\frac{k+\ln_3\mathcal{N}}{\ln_2\mathcal{N}}
	\right)\right).
	\label{eq:block-general-sigma}
\end{align}
Since $\omega(D) \leq J_k$, we say that 
\[
\sum_{\substack{P\in \mathcal{I}_k\\P\mid D}}\frac1{|P|^\sigma} 
\le J_k q^{-\sigma i_k}\ll \frac{u^{-k\sigma} ( \ln_\mathcal{N} ) ^{1-\sigma} (\ln_2\mathcal{N})^{-\sigma} }{k^2 \ln_3\mathcal{N}}
\]
Comparing this with \eqref{eq:block-general-sigma}, we obtain
\[
\frac{\sum_{\substack{P\in \mathcal{I}_k\\P\mid D}} |P|^{-\sigma}}
{\sum_{P\in \mathcal{I}_k}|P|^{-\sigma}}
\ll \frac1{k^2u^k\ln_3\mathcal{N}}.
\]
Therefore, we have 
\begin{align}
\sum_{\substack{P\in \mathcal{I}_k\\
		P\nmid D}}
\frac{1}{|P|^{\sigma}}	={}&\mathcal{Z}(2-\sigma)\ln q \,
u^{k(1-\sigma)}  
(\ln\mathcal{N})^{1-\sigma}
(\ln_2\mathcal{N})^{-\sigma}
\notag\\
&\times\left(1+O\!\left(
\frac{k+\ln_3\mathcal{N}}{\ln_2\mathcal{N}}
\right)+O\!\left(   \frac{1}{k^2 u^k \ln_3 \mathcal{N}} \right)  \right).\label{sum p not divide D}
\end{align}
Now, for $\sigma = 1/2 + \kappa/\ln_2\mathcal{N}$, 
\begin{align*}
	(\ln\mathcal{N})^{1-\sigma}
	&=\sqrt{\ln\mathcal{N}}\,
	\exp\!\left(-\frac{\kappa}{\ln_2\mathcal{N}}\ln_2\mathcal{N}\right)
	=e^{-\kappa}\sqrt{\ln\mathcal{N}},\\
	(\ln_2\mathcal{N})^{-\sigma}
	&=\frac1{\sqrt{\ln_2\mathcal{N}}}
	\exp\!\left(-\frac{\kappa\ln_3\mathcal{N}}{\ln_2\mathcal{N}}\right)
	=\frac{1+o(1)}{\sqrt{\ln_2\mathcal{N}}},\\
	u^{k(1-\sigma)}&=u^{k/2}(1+o(1)),\quad
	u^{1-\sigma}-1=(\sqrt u-1)(1+o(1)),\\
		\mathcal{Z}(2-\sigma) & = \frac1{1-q^{-1/2}\exp( \kappa \ln q/ \ln_2\mathcal{N})} = 	\mathcal{Z}(3/2)(1+o(1)),
\end{align*}
uniformly for $k\le(\ln_2\mathcal{N})^\gamma$.
Therefore, from \eqref{sum p not divide D}, we have
\be
\sum_{\substack{P\in \mathcal{I}_k\\ P\nmid D}}
\frac{1}{|P|^{\sigma}}
=
e^{-\kappa}
\mathcal{Z}\left(\frac32\right)
\ln q\,u^{k/2}(\sqrt{u}-1)
\sqrt{\frac{\ln \mathcal{N}}{\ln_2 \mathcal{N}}}
\left(1+o(1)\right).\label{final sum P not divide D}
\ee
By Stirling's formula,
$
j_k!=\left(\frac{j_k}{e}\right)^{j_k}e^{o(j_k)},
$
substituting \eqref{final sum P not divide D} into \eqref{>>1/j_k!} yields
\[
\left(
\sum_{\substack{
E\mid N_k\\
\omega(E)=j_k\\
(E, D)=1
}}
\frac{1}{{|E|^{\sigma}}}
\right)^2
\gg T_k^{2j_k} e^{o(j_k)},
\]
where 
\[
T_k := \frac{ e^{-\kappa}\mathcal{Z}(\frac{3}{2} ) \,k\,e \ln q \,(\sqrt u -1)   u^{k/2} } {\alpha} \sqrt{\ln_3\mathcal{N}} \left (    1+ o(1)   \right).
\]
Hence, we have 
\begin{equation}\label{final Sk}
    S_{\sigma}(\mathfrak{R}_k)\gg  T_k^{2j_k} e^{o(j_k)}U_k ,
\end{equation}
with
\begin{align*}
U_k & := \sum_{\substack{D |N_k \\ \omega(D) \leq J_k-j_k}} 1\\
& \geq \frac{1}{2} |\mathfrak{R}_k| \left(\frac{J_k}{P_k} \right)^{j_k}, \quad \text{as} \quad \left( |\mathfrak{R}_k|\leq 2\binom{P_k}{J_k} \right).
\end{align*}
Therefore, it follows from \eqref{final Sk} that 
\[
\frac{ S_\sigma(\mathfrak{R}_k)}{|\mathfrak{R}_k|}\gg \left( T_k^2 \frac{J_k}{P_k}\right)^{j_k}e^{o(j_k)} = \left ( \frac{e^{-2\kappa}h}{\alpha^2} \right)^{j_k} e^{o(j_k)},
\]
where 
\[
h:= \frac{\mathcal{Z}(\frac{3}{2})^2 \, e^2 \,c \,\ln q }{\mathcal{Z}(2)} \frac{\sqrt u -1}{\sqrt u +1}.
\]
Using the above estimate in \eqref{SR}, we obtain
\begin{equation}\label{SR/R}
    \frac{ S_\sigma(\mathfrak{R})}{|\mathfrak{R}|}\geq \exp \left ( (\beta +o(1) \sqrt{\frac{\ln \mathcal{N} \ln_3 \mathcal{N}}{\ln_2 \mathcal{N}}} \right),
\end{equation}
where $\beta = \alpha \gamma  \ln(e^{-2\kappa}h/\alpha^2)$. Choosing the optimal value $\alpha=e^{-\kappa}\sqrt{h}/e$, we obtain
\[
\beta=\frac{2e^{-\kappa}\gamma\sqrt{h}}{e}= 2e^{-\kappa} \gamma \mathcal{Z}\left ( \frac {3}{2} \right) \sqrt{ \frac{c \ln q }{\mathcal{Z}(2)}\frac{\sqrt u-1}{\sqrt u +1}}.
\]
By letting $c \gamma \ln u \to 1$ and then $u\to 1$, we see that $\beta$ can be made arbitrarily close to
\[
e^{-\kappa}\mathcal{Z}\left ( \frac {3}{2} \right) \sqrt{\frac{\ln q}{\mathcal{Z}(2)}} = e^{-\kappa}\sqrt{ \frac{\sqrt q+1}{\sqrt q -1} \ln q   }.
\]
This completes the proof of Proposition \ref{prop}.

\section{Proof of  Theorem \ref{main thm}}\label{proof thm 1}

Recall that, by the construction of the set $\mathfrak{R}$ (see \eqref{set R}), every polynomial $R\in\mathfrak{R}$ is squarefree and, for each $1\leq k\leq (\ln_2\mathcal{N})^\gamma$, its $k$-th component $R_k\in\mathfrak{R}_k$ contains at most
$\frac{c \ln \mathcal{N}}{k^2 \ln_3 \mathcal{N}}$
irreducible factors, all of which belong to  $\mathcal{I}_k\cap\mathcal{P}$ and $|\mathfrak{R}| \leq \mathcal{N}$ from Lemma \ref{card R}. Note that every  $H \in \mathfrak{R}$ satisfies
\[
\mathbf{d}(H)\ll \ln \mathcal{N}\ln_2 \mathcal{N}.
\]
For fixed  $P \in \mathcal{P}_{2g+1}$, define the resonator 
\[
\mathcal{R}_{P}=\sum_{H\in \mathfrak{R}} \chi_{P}(H).
\]
Let $\sigma= 1/2+\sigma_\kappa$, where $\sigma_\kappa =\kappa/\ln g$. Also assume that  $\mathcal{N} \asymp q^{g(\sigma - \epsilon')}$, where $0<\epsilon'<1/2$. 
 Then, for the resonance method, we consider the following two sums:
 \begin{equation*}
 	\mathcal{S}_1 := \sum_{P \in \mathcal{P}_{2g+1}}\mathcal{R}_P^2,\qquad \text{and} \qquad \mathcal{S}_2 := \sum_{P \in \mathcal{P}_{2g+1}} L\left(\sigma, \chi_P\right)\mathcal{R}_P^2   .
 \end{equation*}
Therefore, we have 
\begin{equation}\label{max L}
\max_{P \in \mathcal{P}_{2g+1}} \left| L(\sigma, \chi_P) \right| \ge \frac{\mathcal{S}_2}{\mathcal{S}_1}.
\end{equation}
It remains to establish a lower bound for $\mathcal{S}_2$ and an upper bound for $\mathcal{S}_1$. We begin with the lower bound for $\mathcal{S}_2$. By Lemma \ref{ApproximateFE}, we get
\begin{align*}
\mathcal{S}_2 &= \sum_{P \in \mathcal{P}_{2g+1}} L\left(\sigma, \chi_P\right) \mathcal{R}_{P}^2 \nonumber \\
&= \sum_{M,N \in \mathfrak{R}}
   \sum_{K \in \mathcal{M}_{\le g}} \frac{1}{{|K|^{\sigma}}}
   \sum_{P \in \mathcal{P}_{2g+1}} \chi_P(K M N) \nonumber 
\quad + q^{g(1-2\sigma)} \sum_{M,N \in \mathfrak{R}}
   \sum_{K \in \mathcal{M}_{\le g-1}} \frac{1}{{|K|^{1-\sigma}}}
   \sum_{P \in \mathcal{P}_{2g+1}} \chi_P(K M N) \nonumber \\
&=:\mathcal{S}_{21} + \mathcal{S}_{22}.
\end{align*}
We now split the innermost character sum, following Darbar and Maiti \cite{darbar2024}, according to whether $KMN$ is a perfect square. Applying Lemma \ref{char sum} to the square contribution and Lemma \ref{weil} to the non-square contribution, we obtain
\begin{align}
\mathcal{S}_{21} &= |\mathcal{P}_{2g+1}| \sum_{M,N \in \mathfrak{R}}
\sum_{\substack{K \in \mathcal{M}_{\le g} \\ KMN= \square}} \frac{1}{{|K|}^{1/2+\sigma_\kappa}}
+ O\!\left( \frac{q^g}{g} \sum_{\substack{K \in \mathcal{M}_{\le g} \\ M,N \in \mathfrak{R} \\ KMN \neq \square}} 
\frac{\mathbf{d}(KMN)}{{|K|}^{1/2+\sigma_\kappa}} \right), \nonumber \\
\mathcal{S}_{22} &=   q^{-2g \sigma_\kappa}   |\mathcal{P}_{2g+1}| \sum_{M,N \in \mathfrak{R}}
\sum_{\substack{K \in \mathcal{M}_{\le g-1} \\ KMN= \square}} \frac{1}{|K|^{1/2-\sigma_\kappa}}
+ O\!\left( q^{-2g\sigma_\kappa} \frac{q^g}{g} \sum_{\substack{K \in \mathcal{M}_{\le g-1} \\ M,N \in \mathfrak{R} \\ KMN \neq \square}} 
\frac{\mathbf{d}(KMN)}{|K|^{1/2-\sigma_\kappa}} \right). \nonumber
\end{align}
We now estimate $\mathcal{S}_{21}$.  The bound
\[
\mathbf{d}(KMN) \ll g+\ln \mathcal{N}\ln_2 \mathcal{N},
\]
and $$\sum_{K\in \mathcal{M}_{\leq g}} \frac{1}{|K|^{1/2+\sigma_\kappa}}  = \sum_{d=0}^{g}  q^{d(1/2-\sigma_\kappa)} \ll q^{g(1/2-\sigma_\kappa)}, $$
allows us to control the non-square contribution due to Lemma~\ref{weil}. Therefore, 
\begin{equation*}
\mathcal{S}_{21} = |\mathcal{P}_{2g+1}| \sum_{M,N \in \mathfrak{R}}
\sum_{\substack{K \in \mathcal{M}_{\le g} \\ KMN = \square}} \frac{1}{|K|^{1/2+\sigma_\kappa}}
+ O\!\left( q^{3g/2-g\sigma_\kappa} \mathcal{N}^2 (g + \ln \mathcal{N} \ln_2 \mathcal{N}) \right).
\end{equation*}
Using Lemma \ref{prime pol thm}, we have 
\begin{equation}\label{S21}
\mathcal{S}_{21} = |\mathcal{P}_{2g+1}| \left(  \sum_{M,N \in \mathfrak{R}}
\sum_{\substack{K \in \mathcal{M}_{\le g} \\ KMN = \square}} \frac{1}{|K|^{1/2+\sigma_\kappa}}
+ O\!\left( \frac {\mathcal{N}^2 (g + \ln \mathcal{N} \ln_2 \mathcal{N})}{q^{g\sigma}} \right) \right).
\end{equation}
Let the main term be denoted by
\[
\mathcal{D}_\sigma
:=
\sum_{M,N\in\mathfrak{R}}
\sum_{\substack{K\in\mathcal{M}_{\leq g}\\ KMN=\square}}
\frac{1}{|K|^{\sigma}}.
\]
Since all the terms under consideration are non-negative, we may restrict
the sum to the contribution corresponding to
\[
K=\frac{[M,N]}{(M,N)},
\]
whenever $d(K)\leq g$. Indeed, for this choice we have
\[
KMN=\frac{[M,N]}{(M,N)}\,MN=[M,N]^2,
\]
and hence $KMN$ is a perfect square. Moreover, the contribution from
those $M,N\in\mathfrak{R}$ for which
$d([M,N]/(M,N))>g$ is at most
\[
\mathcal{N}^2q^{-\sigma g}
=o\bigl(S_\sigma(\mathfrak{R})\bigr).
\]
Therefore,
\begin{align*}
	\mathcal{D}_\sigma
	\geq (1+o(1))
	\sum_{M,N\in\mathfrak{R}}
	\frac{|(M,N)|^\sigma}{|[M,N]|^\sigma}.
\end{align*}
It follows from Proposition~\ref{prop} that
\begin{equation}\label{D}
	\mathcal{D}_\sigma \geq \mathcal{N}
	\exp\left(
	\left(
	e^{-\kappa}
	\sqrt{\frac{\sqrt{q}+1}{\sqrt{q}-1}\ln q}
	+o(1)
	\right)
	\sqrt{\frac{\ln\mathcal{N}\ln_3\mathcal{N}}
		{\ln_2\mathcal{N}}}
	\right).
\end{equation}
Here we have used $\ln_2\mathcal{N}=\ln g+O(1)$, so that the
parameter in Proposition~\ref{prop} is
\[
\kappa\,\frac{\ln_2\mathcal{N}}{\ln g}
=\kappa+o(1).
\]
Hence, substituting \eqref{D} into \eqref{S21}, we get
\begin{equation*}
    \mathcal{S}_{21}  \geq |\mathcal{P}_{2g+1}| \left(  
       \mathcal{N} \exp\left( \left(e^{-\kappa}\sqrt{\frac{\sqrt{q}+1}{\sqrt{q}-1}\ln q}+o(1)\right)\sqrt{\frac{\ln\mathcal{N} \ln_3 \mathcal{N}}{\ln_2 \mathcal{N}}}\right)   
+O\!\left( \frac {\mathcal{N}^2 (g + \ln \mathcal{N} \ln_2 \mathcal{N})}{q^{g\sigma}} \right) \right).
\end{equation*}
Again, since $\sum_{f\in \mathcal{M}_{\leq g-1}} \frac{1}{|f|^{1/2-\sigma_\kappa}}  = \sum_{d=0}^{g}  q^{d(1/2+\sigma_\kappa)} \ll q^{g(1/2+\sigma_\kappa)} = q^{g\sigma}$, $\sigma+1-2\sigma_\kappa=3/2-\sigma_\kappa$ and calculations similar to $\mathcal{S}_{21}$, we have
\begin{equation*}
	\mathcal{S}_{22}  \geq  q^{-2g\sigma_k} |\mathcal{P}_{2g+1}|  \sum_{M,N \in \mathfrak{R}}
	\sum_{\substack{K \in \mathcal{M}_{\le g-1} \\ KMN= \square}} \frac{1}{|K|^{1/2-\sigma_\kappa}}     
	+O\!\left( \frac {\mathcal{N}^2 (g + \ln \mathcal{N} \ln_2 \mathcal{N})}{q^{g\sigma}} |\mathcal{P}_{2g+1}|\right).
\end{equation*}
Now, for any $K \in \mathcal{M}$, we say that $ |K|^{2\sigma_\kappa}\geq 1$, so 
\begin{equation*}
	\mathcal{S}_{22}  \geq  q^{-2g\sigma_k} |\mathcal{P}_{2g+1}|  \sum_{M,N \in \mathfrak{R}}
	\sum_{\substack{K \in \mathcal{M}_{\le g-1} \\ KMN= \square}} \frac{1}{|K|^{1/2+\sigma_\kappa}}     
	+O\!\left( \frac {\mathcal{N}^2 (g + \ln \mathcal{N} \ln_2 \mathcal{N})}{q^{g\sigma}} |\mathcal{P}_{2g+1}|\right).
\end{equation*}
Hence, similarly to $\mathcal{S}_{21}$, we have 
\begin{equation*}
	\mathcal{S}_{22}  \geq |\mathcal{P}_{2g+1}| \left(  q^{-2g\sigma_\kappa}
	\mathcal{N} \exp\left( \left(e^{-\kappa}\sqrt{\frac{\sqrt{q}+1}{\sqrt{q}-1}\ln q}+o(1)\right)\sqrt{\frac{\ln\mathcal{N} \ln_3 \mathcal{N}}{\ln_2 \mathcal{N}}}\right)   
	+O\!\left( \frac {\mathcal{N}^2 (g + \ln \mathcal{N} \ln_2 \mathcal{N})}{q^{g\sigma}} \right) \right).
\end{equation*}
Combining the estimates for $\mathcal {S}_{21}$ and $\mathcal {S}_{22}$, we obtain
\begin{align*}
\mathcal {S}_{2}
\geq &
|\mathcal P_{2g+1}|
\left(1+q^{-2\kappa g/\ln g}\right)
\mathcal{N} \exp\left( \left(e^{-\kappa}\sqrt{\frac{\sqrt{q}+1}{\sqrt{q}-1}\ln q}+o(1)\right)\sqrt{\frac{\ln\mathcal{N} \ln_3 \mathcal{N}}{\ln_2 \mathcal{N}}}\right)\\
&
+O\!\left( \frac {\mathcal{N}^2 (g + \ln \mathcal{N} \ln_2 \mathcal{N})}{q^{g\sigma}} |\mathcal{P}_{2g+1}|\right).
\end{align*}
For $\kappa\geq0$, 
$
1+q^{-2\kappa g/\ln g}\geq1 
$
and  for the choice
$
\mathcal{N} \asymp q^{g(\sigma-\epsilon')},
$
where \(0<\epsilon'<1/2\) is arbitrarily small, we have
\begin{equation}\label{final est S2}
	\mathcal{S}_{2}  \gg  (1+o(1)) \mathcal{N}|\mathcal{P}_{2g+1}| \exp\left( e^{-\kappa} \left(\sqrt{\frac{\sqrt{q}+1}{\sqrt{q}-1}\ln q}+o(1)\right)\sqrt{\frac{\ln\mathcal{N} \ln_3 \mathcal{N}}{\ln_2 \mathcal{N}}}\right).
\end{equation}
We now turn to the derivation of an upper bound for $\mathcal{S}_1$. Since $\mathfrak{R}$ consists of squarefree polynomials, the condition $MN=\square$
 forces $M=N$. Therefore, following the preceding arguments, we obtain
\begin{align*}
\mathcal{S}_1 &= \sum_{P \in \mathcal{P}_{2g+1}} \mathcal{R}_P^2=\sum_{M,N \in \mathfrak{R}_{}} \sum_{P \in \mathcal{P}_{2g+1}} \chi_{P}(MN) \\
 &= |\mathcal{P}_{2g+1}|\sum_{\substack{M,N \in \mathfrak{R} \\  MN = \square}} 1 +  O\!\left( \frac{q^g}{g} \sum_{\substack{ M,N \in \mathfrak{R} \\  MN \neq \square}} 
\mathbf{d}(MN) \right)\\
& \leq |\mathcal{P}_{2g+1}| \left(
 \mathcal{N} 
+ O\!\left( \frac {\mathcal{N}^2 (g + \ln \mathcal{N} \ln_2 \mathcal{N})}{q^{g}} \right) \right).
\end{align*}
With the above choice
$
\mathcal{N} \asymp q^{g(\sigma-\epsilon')},
$
it follows that
\begin{equation}\label{final esti S1}
\mathcal{S}_1 \le (1+o(1))\,\mathcal{N}\,|\mathcal{P}_{2g+1}|.	
\end{equation}
Consequently, for  arbitrarily small $\epsilon'>0$ and  sufficiently large $g$ , we obtain 
\begin{align*}
    \max_{P \in \mathcal{P}_{2g+1}} \left| L(\sigma, \chi_P) \right| \ge \frac{\mathcal{S}_2}{\mathcal{S}_1} \gg \exp\left( \left(e^{-\kappa}\sqrt{\frac{\sqrt{q}+1}{\sqrt{q}-1}\ln q}+o(1)\right)\sqrt{\frac{\ln\mathcal{N} \ln_3 \mathcal{N}}{\ln_2 \mathcal{N}}}\right) \\
    =\exp \left( \left( e^{-\kappa} \sqrt{\frac{(1/2+\sigma_\kappa-\epsilon')(\sqrt{q}+1)}{\sqrt{q}-1} } \, \ln q + o(1) \right) 
\sqrt{\frac{g \ln_2 g}{\ln g}} \right).
\end{align*}
For sufficiently large $g$, 
\[
  \sqrt { \left ( \frac{1}{2} +\sigma_\kappa -\epsilon' \right)}= \sqrt{\left (\frac{1}{2}-\epsilon'\right) }\,+o(1).
\]
This concludes the proof of Theorem~\ref{main thm}.

\section{Proof of Theorem \ref{thm inside critical strip}}\label{section proof inside critical strip}

Let  $P\in\mathbb{F}_q[t]$ be a monic irreducible polynomial with $ \mathbf{d}(P)=\mathbf{N}$ and fix $1/2<\sigma<1$.
In this section, we take $Q$ to be a monic  irreducible. For $\delta >0$, taking $Y_o= \frac{1+\delta}{\sigma-1/2} \log \mathbf{N}$, from the following equation of \cite[equation 2.12]{lumley2021}, we get
\begin{equation}\label{log L eq}
	\ln L(\sigma,\chi_P) = \sum_{\substack{A\ {\rm monic}\\ \deg A\le Y_o}} \frac{\Lambda_{\mathcal{A}}(A) \chi_P(A)}{\mathbf{d}(A) |A|^{\sigma}}+O\left(  \frac{1}{\mathbf{N}^{\delta}} \right).
\end{equation}
Since for irreducible polynomial $Q$, $\sum_{Q}\frac{1}{|Q|^{2\sigma}}\ll\sum_{\mathbf{d}(Q)\geq1} \frac{q^{(1-2\sigma)\mathbf{d}(Q)}}{\mathbf{d}(Q)} \ll 1$, we obtain 
\[
\sum_{\substack{
		Q \ \mathrm{irreducible}\\
		k\geq 2,\; k\,\mathbf{d}(Q)\leq Y_o
}}
\frac{\chi_P(Q^k)}{k\,|Q|^{k\sigma}}
\ll 1.
\]
Therefore, 
\begin{equation}\label{Log L=}
	\ln L(\sigma,\chi_P)= S_P(\sigma,Y_o)+O(1),
\end{equation}
where
\[
S_P(\sigma,Y_o)=\sum_{\mathbf{d}(Q)\leq Y_o}\frac{\chi_P(Q)}{|Q|^{\sigma}}.
\]
Therefore, the proof of Theorem~\ref{thm inside critical strip} reduces to showing that   $S_P(\sigma,Y_o)$ takes large values. To obtain large values of $S_P(\sigma,Y_o)$, we use the long resonator method as in \cite{darbar2025large}.

We choose $Y$ such that $ a/q \,\,  \mathbf{N} \log \mathbf{N}\leq q^{Y} < a \mathbf{N} \log \mathbf{N} $, where $a>0$ is a parameter that we will choose optimally later. Moreover, we set $r(1)=1$ and, for each monic  irreducible polynomial $Q$, define
$r(Q)=b$ when $\mathbf{d}(Q)\leq Y$ and $r(Q)=0$ otherwise, with fixed $0<b<1$. We extend these weights  to all monic polynomials $M\in\mathcal{M}$ by complete multiplicativity. We define the resonator
\[
\mathcal{R}_P = \prod_{\mathbf{d}(Q) \leq Y} \frac{1}{1-r(Q)\chi_P(Q)}= \sum_{M \, \, \text{monic}} r(M)\chi_P(M).
\]
We again define the following two sums:
\begin{equation*}
	\mathcal{S}_1(\sigma):= \sum_{P \in \mathcal{P}_{\mathbf{N}}} \mathcal{R}_P ^2, \qquad \text{and} \qquad \mathcal{S}_2(\sigma):= \sum_{P \in \mathcal{P}_{\mathbf{N}}}S_P(\sigma,Y_o) \mathcal{R}_P ^2.
\end{equation*}
Therefore, we have 
\begin{equation}\label{max Sp}
	\max_{P\in\mathcal{P}_{\mathbf{N}}}
		S_P(\sigma,Y_o)
	\geq
	\frac{\mathcal{S}_2(\sigma)}{\mathcal{S}_1(\sigma)}.
\end{equation}
It remains to establish a lower bound of $\mathcal{S}_2(\sigma)$ and an upper bound of $\mathcal{S}_1(\sigma)$. We have
\begin{equation*}
	\mathcal{S}_2(\sigma)= \sum_{M,N}r(M)r(N) \sum_{\mathbf{d}(Q) \leq Y_o} \frac{1}{|Q|^{\sigma}} \sum_{P \in \mathcal{P}_{\mathbf{N}}}\chi_P(MNQ).
\end{equation*}
Applying Lemma \ref{weil} and  Lemma \ref{char sum}, we obtain
\begin{align}\label{S2(sigma)}
	\mathcal{S}_2(\sigma)= & |\mathcal{P}_{\mathbf{N}}|  \sum_{\mathbf{d}(Q) \leq Y_o}  \frac{1}{|Q|^{\sigma}} \sum_{\substack{
			M,N\\
			MNQ=\square
	}}r(M)r(N)  + O \left ( \frac{q^{\mathbf{N/2}}}{\mathbf{N}}  \sum_{\mathbf{d}(Q) \leq Y_o}  \frac{1}{|Q|^{\sigma}} \sum_{\substack{
			M,N\\MNQ\ne \square}}  \mathbf{d}(MNQ)r(M)r(N) \right)\nonumber  \\
	& = : \mathcal{S}_2^m+ \mathcal{S}_2^o.
\end{align}
The square condition $MNQ=\square$ implies $\deg Q\leq Y$. Moreover,
by positivity of $r(M)$, we may restrict attention to the contribution from those
irreducible polynomials $Q$ satisfying $Q\mid N$. Therefore,
\begin{align*}
	\mathcal{S}_2^m & =  |\mathcal{P}_{\mathbf{N}}|   \sum_{\mathbf{d}(Q) \leq Y_o} \sum_{\substack{
			M,N\\MNQ = \square}}  \frac{r(M)r(N)}{|Q|^{\sigma}} 
	\\
	&
	\geq |\mathcal{P}_{\mathbf{N}}|\sum_{\mathbf{d}(Q) \leq Y} \sum_{\substack{
			M,N \\ MNQ= \square,\, \,  Q \mid N}}  \frac{r(M)r(N)}{|Q|^{\sigma}} 
	\\
	&
	\geq  |\mathcal{P}_{\mathbf{N}}|  \sum_{\mathbf{d}(Q) \leq Y} \frac{r(Q)}{|Q|^{\sigma}}       \sum_{\substack{
			M,L \, \text{monic} \\ ML= \square}} r(M)r(L),
\end{align*}
where we write $N=QL$. By our  choice of $Y$ and using Lemma \ref{produc of poly}, we get 
\begin{align}\label{6. >=}
	\sum_{\mathbf{d}(Q) \leq Y}    \frac{r(Q)}{|Q|^{\sigma}}      =    b \sum_{\mathbf{d}(Q)       \leq Y}      \frac{1}{|Q|^{\sigma}} =  & \, b \left(1+O(\ln Y/Y)\right)    \mathcal{Z}(2-\sigma)     \frac{q^{Y(1-\sigma)}}{Y} \nonumber
	\\
	\geq & \, b \, (1+o(1)) \left(\frac{a}{q}\right)^{1-\sigma} \mathcal{Z}(2-\sigma) \frac{\mathbf{N}^{1-\sigma}}{ (\log \mathbf{N})^{\sigma} }.
\end{align}
We now estimate the error term $\mathcal{S}_2^o$. First, we have  
\[
\sum_{M}r(M) =\prod_{\mathbf{d}(Q)\leq Y}(1+b+b^2+...)=\frac{1}{(1-b)^{\Pi_q(Y)}},
\]
and 
\[
\sum_{M,N} \mathbf{d}(MNQ)r(M)r(N)=2 \sum_{M}d(M)r(M) \sum_N r(N)+ \mathbf{d}(Q)\left (\sum_{M}r(M) \right)^2.
\]
Moreover, since $r$ is completely multiplicative and $r(Q)=b$ for $d(Q)\le Y$ and $\sum_{\mathbf{d}(Q)\leq Y}\mathbf{d}(Q)\ll\sum_{y\leq Y }y \, q^y/y \ll q^Y$, we obtain
\[
\sum_{M}\mathbf{d}(M)r(M)= \frac{b}{1-b} \sum_{M}r(M)\, \sum_{\mathbf{d}(Q)\leq Y} \mathbf{d}(Q)\ll q^Y  \sum_{M}r(M).
\]
Combining these estimates, we therefore get
\begin{align}\label{error of S2}
	\mathcal{S}_2^o \ll \frac{q^{\mathbf{N}/2}}{\mathbf{N}} (1-b)^{-2 \Pi_q(Y)}      \left(   q^Y \sum_{\mathbf{d}(Q)\leq Y_o} \frac{1}{|Q|^{\sigma}} +  \sum_{\mathbf{d}(Q)\leq Y_o} \frac{\mathbf{d}(Q)}{|Q|^{\sigma}}\right). 
\end{align}
Since $\sum_{\mathbf{d}(Q)\leq Y_o}1/|Q|^{\sigma} \ll q^{Y_o(1-\sigma)}/Y_o$, it follows from \eqref{error of S2} that
\begin{align*}
	\mathcal{S}_2^o   \ll       \frac{q^{\mathbf{N}/2}}{\mathbf{N}}   (1-b)^{-2 \Pi_q(Y)} \frac{q^{Y_o(1-\sigma)}}{Y_o}q^{Y}=     \frac{  q^{\mathbf{N}/2+Y+Y_o(1-\sigma)}  }   {\mathbf{N}Y_o  }  (1-b)^{-2 \Pi_q(Y)}.
\end{align*}
Therefore, combining \eqref{6. >=} and \eqref{error of S2}, we obtain 
\begin{align}\label{final S2sigma lower bound}
	\mathcal{S}_2(\sigma) \geq  & |\mathcal{P}_{\mathbf{N}}|   \,    b \, (1+o(1)) \left(\frac{a}{q}\right)^{1-\sigma} \mathcal{Z}(2-\sigma)  \sum_{\substack{
			M,L \, \text{monic} \\ ML= \square}} r(M)r(L)    \, \,  \frac{\mathbf{N}^{1-\sigma}}{ (\log \mathbf{N} )^{\sigma}}\nonumber
	\\
	&
	+ O\left(   \frac{  q^{\mathbf{N}/2+Y+Y_o(1-\sigma)}  }   {\mathbf{N}Y_o  }  (1-b)^{-2 \Pi_q(Y)}    \right).
\end{align}
Similarly, by applying Lemmas~\ref{weil} and~\ref{char sum} once again and proceeding as in the estimation of $\mathcal{S}_2(\sigma)$, we obtain
\begin{align}
	\mathcal{S}_1(\sigma)= &  \sum_{P \in \mathcal{P}_{\mathbf{N}}} \mathcal{R}_P ^2= \sum_{M,N}r(M)r(N)  \sum_{P \in \mathcal{P}_{\mathbf{N}}}\chi_P(MN)\nonumber \\
	= & |\mathcal{P}_{\mathbf{N}}| \sum_{\substack{
			M,N\\ MN=\square }}r(M)r(N)+   O \left ( \frac{q^{\mathbf{N/2}}}{\mathbf{N}}  \sum_{\substack{
			M,N\\MN\ne \square}}  \mathbf{d}(MN)r(M)r(N) \right) \nonumber\\
	= & |\mathcal{P}_{\mathbf{N}}| \sum_{\substack{
			M,N\\ MN=\square }}r(M)r(N) + O\left(   \frac{  q^{\mathbf{N}/2+Y}  }   {\mathbf{N}  }  (1-b)^{-2 \Pi_q(Y)}    \right).\label{S1=|PN|}
\end{align}
To find the optimal value of $a$, we evaluate the resonator sum as follows,
\begin{align}\label{reso sum}
	\sum_{\substack{M,N\\ MN=\square }}r(M)r(N)  =  & \prod_{\deg Q\le Y}
	\frac{1+r(Q)^2}{(1-r(Q)^2)^2} =\left(
	\frac{1+b^2}{(1-b^2)^2}
	\right)^{\Pi_q(Y)} \nonumber\\
	= & \exp \left ( (1+o(1)) q^Y /Y \lambda(b)   \right)  \nonumber\\
	\leq & \exp \left ( ( a  \lambda(b)\mathcal{Z}(2)+o(1))    \mathbf{N} \right),
\end{align}
where $\lambda(b):= \ln \left ( \frac{1+b^2}{(1-b^2)^2} \right)$. From the choice of $Y$, it follows that $Y=O(\log \mathbf{N})=o(\mathbf{N})$, so that 
\begin{align}\label{eror}
	\frac{  q^{\mathbf{N}/2+Y}  }   {\mathbf{N}  }  (1-b)^{-2 \Pi_q(Y)}=  & \frac{  q^{\mathbf{N}/2}  }   {\mathbf{N}  } \exp \left (   -2 a \ln (1-b) \mathcal{Z}(2) \mathbf{N}
	(1+o(1)  ) +o(\mathbf{N})\right) \nonumber\\
	= &  \frac{  q^{\mathbf{N}/2}  }   {\mathbf{N}  }   \exp \left ( -2 a \ln (1-b) \mathcal{Z}(2) \mathbf{N}
	(1+o(1))\right).  
\end{align}
Therefore, combining the estimates for $\mathcal{S}_1(\sigma)$ and $\mathcal{S}_2(\sigma)$, and comparing the corresponding main and error terms, we obtain
\begin{equation}\label{s2/s1}
	\frac{\mathcal{S}_2(\sigma)}{\mathcal{S}_1(\sigma)} \geq  b \,  \left(\frac{a}{q}\right)^{1-\sigma} \mathcal{Z}(2-\sigma) (1+o(1)) \frac{\mathbf{N}^{1-\sigma}}{ (\log \mathbf{N} )^{\sigma}},
\end{equation}
provided that 
\begin{align*}
	& \frac{1}{2} \ln q  -2a \ln(1-b) \mathcal{Z}(2) <\ln q+ a \lambda(b)\mathcal{Z} (2)\\
	& \Longleftrightarrow \, \,  a< \frac{\ln q}{2 \mathcal{Z}(2) } \left (  \ln \left (   \frac{(1+b)^2}{1+b^2} \right)  \right)^{-1},
\end{align*}
where we have used \eqref{reso sum} and \eqref{eror}.
Taking $a=\frac{\ln q}{2 \mathcal{Z}(2) } \left (  \ln \left (   \frac{(1+b)^2}{1+b^2} \right)  \right)^{-1}-\xi$  with arbitrarily small $\xi>0$, and combining \eqref{Log L=}, \eqref{max Sp}, and \eqref{s2/s1}, we obtain the first part of Theorem~\ref{thm inside critical strip}.

For the proof of the second part of Theorem~\ref{thm inside critical strip}, we choose 
\begin{equation*}
	a
	=
	\frac{1}{\alpha_q(b)}
	-\xi+\frac{1}{ \sqrt{\log \mathbf{N}} },
	\label{eq:a-choice}
\end{equation*}
with  $1/ \sqrt{\log \mathbf{N}} <\xi$. 
Therefore,
\begin{align}
	a^{1-\sigma}
	&=
	\alpha_q(b)^{\sigma-1} (1-\xi\alpha_q(b))^{1-\sigma}
	\left(1+\frac{\alpha_q(b)}{(1-\xi\alpha_q(b)) \sqrt{\log \mathbf{N}} }
	\right)^{1-\sigma}\nonumber\\
	& \geq \alpha_q(b)^{\sigma-1} (1-\xi\alpha_q(b))^{1-\sigma} \left(1+\frac{(1-\sigma)\alpha_q(b)}{2(1-\xi\alpha_q(b)) \sqrt{\log \mathbf{N}} }
	\right).
	\label{eq:a-power}
\end{align}

Substituting \eqref{eq:a-power} into
\eqref{s2/s1}, we obtain
\begin{align}
	\frac{\mathcal S_2(\sigma)}
	{\mathcal S_1(\sigma)}
	\geq\;&
	bq^{\sigma-1}\mathcal Z(2-\sigma)
	\alpha_q(b)^{\sigma-1}
	\left(
	1-\xi\alpha_q(b)
	\right)^{1-\sigma} 
	\notag\\
	&\times \left(1+\frac{(1-\sigma)\alpha_q(b)}{2(1-\xi\alpha_q(b)) \sqrt{\log \mathbf{N}} }
	\right)
	(1+o(1))
	\frac{\mathbf N^{1-\sigma}}
	{(\log\mathbf N)^\sigma} \notag\\
	\geq 	&  \left( \tau_{\xi,\mathbf N}+     \frac{b q^{\sigma-1} \mathcal{Z}(2-\sigma)(1-\sigma)\alpha_q(b)^{\sigma} \mathbf N^{1-\sigma} }{4(1-\xi\alpha_q(b))^{\sigma}  ( \log \mathbf{N})^{\sigma+1/2}  }   \right)
	\label{eq:ratio-delta}.
\end{align}
By using \eqref{Log L=}, we  get 
\begin{align}\label{sum log L}
	\sum_{P\in\mathcal P_{\mathbf N}}
	\ln L(\sigma,\chi_P)R_P^2  & \geq
	\left( \tau_{\xi,\mathbf N}+     \frac{b q^{\sigma-1} \mathcal{Z}(2-\sigma)(1-\sigma)\alpha_q(b)^{\sigma} \mathbf N^{1-\sigma} }{4(1-\xi\alpha_q(b))^{\sigma}  ( \log \mathbf{N})^{\sigma+1/2}  }   \right)   \sum_{P\in\mathcal P_{\mathbf N}}R_P^2.
\end{align}
We divide the sum on the left-hand side into two parts, according to whether $\ln L(\sigma,\chi_P)$ is less than or greater than $\tau_{\xi,\mathbf N}$. Thus,
\begin{align*}
	\sum_{P\in\mathcal P_{\mathbf N}}
	\ln L(\sigma,\chi_P)R_P^2 &
	=
	\sum_{\substack{P\in\mathcal P_{\mathbf N}
			\\	\ln L(\sigma,\chi_P)<\tau_{\xi,\mathbf N}}}
	\ln L(\sigma,\chi_P)R_P^2
	+
	\sum_{\substack{P\in\mathcal P_{\mathbf N}\\
			\ln L(\sigma,\chi_P)\geq\tau_{\xi,\mathbf N}}}
	\ln L(\sigma,\chi_P)R_P^2
	\notag \\
	& <
	\tau_{\xi,\mathbf N}
	\sum_{P\in\mathcal P_{\mathbf N}}R_P^2
	+
	\sum_{\substack{P\in\mathcal P_{\mathbf N}\\
			\ln L(\sigma,\chi_P)\geq\tau_{\xi,\mathbf N}}}
	\ln L(\sigma,\chi_P)R_P^2.
\end{align*}
Therefore, from \eqref{sum log L}, we obtain
\begin{align}
	\sum_{\substack{P\in\mathcal P_{\mathbf N}\\
			\ln L(\sigma,\chi_P)\geq\tau_{\xi,\mathbf N}}}
	\ln L(\sigma,\chi_P)R_P^2  & >       \left ( \frac{b q^{\sigma-1} \mathcal{Z}(2-\sigma)(1-\sigma)\alpha_q(b)^{\sigma} \mathbf N^{1-\sigma} }{4(1-\xi\alpha_q(b))^{\sigma}  ( \log \mathbf{N})^{\sigma+1/2}  }   \right)   \sum_{P\in\mathcal P_{\mathbf N}}R_P^2 \nonumber\\
	& \geq e^{ o(\mathbf{N})}(1+o(1))
	|\mathcal P_{\mathbf N}|  \exp {\left( \Pi_q(Y) \lambda(b)\right)}, \label{log L >=}
\end{align}
by using    
\begin{equation*}
	\mathcal { S}_1(\sigma)
	=
	\sum_{P\in\mathcal P_{\mathbf N}}R_P^2
	=
	(1+o(1))
	|\mathcal P_{\mathbf N}|\exp {\left( \Pi_q(Y) \lambda(b)\right)}.
\end{equation*}
Recall that 
\[
R_P
=
\prod_{\deg Q\leq Y}
(1-b\chi_P(Q))^{-1}.
\]
Hence
\begin{equation}
	R_P^2\leq(1-b)^{-2\Pi_q(Y)}.
	\label{eq:Rmax}
\end{equation}
Moreover, from \eqref{Log L=},
\begin{align*}
	|\ln L(\sigma,\chi_P)|
	&\ll
	\sum_{\deg Q\leq Y_o}
	\frac{1}{|Q|^\sigma}+1\ll
	\frac{q^{(1-\sigma)Y_o}}{Y_o}+1.
\end{align*}
Since $Y_o= \frac{1+\delta}{\sigma-1/2} \log \mathbf{N}=O( \log \mathbf{N} )=o(\mathbf{N})$,  it follows that 
\begin{equation}
	\max_{P\in\mathcal P_{\mathbf N}}
	|\ln L(\sigma,\chi_P)|
	\leq e ^ {o(\mathbf N)}.
	\label{eq:Lmax}
\end{equation}
Furthermore,	
\begin{align}
	\sum_{\substack{P\in\mathcal P_{\mathbf N}\\
			\ln L(\sigma,\chi_P)\geq\tau_{\xi,\mathbf N}}}
	\ln L(\sigma,\chi_P)R_P^2  \, \leq \, 
	\max_{P\in\mathcal P_{\mathbf N}}
	|\ln L(\sigma,\chi_P)|     	\max_{P\in\mathcal P_{\mathbf N}} R_P^2 \,\,
	\Phi_{\mathbf N}(\sigma,\xi)
	|\mathcal P_{\mathbf N}|.
	\label{eq:large-upper}
\end{align} 
Consequently, using \eqref{eq:Rmax} and \eqref{eq:Lmax} in \eqref{eq:large-upper}, we obtain
\begin{align}
	\sum_{\substack{P\in\mathcal P_{\mathbf N}\\
			\ln L(\sigma,\chi_P)\geq\tau_{\xi,\mathbf N}}}
	\ln L(\sigma,\chi_P)R_P^2 \, \leq \, 
	|\mathcal P_{\mathbf N}|
	\Phi_{\mathbf N}(\sigma,\xi)
	e^{o(\mathbf N)}
	(1-b)^{-2 \Pi_q(Y)}.
	\label{eq:large-upper1}
\end{align}
Comparing \eqref{log L >=} with \eqref{eq:large-upper1}, we get
\begin{align}\label{phi> exp(-piQY)}
	\Phi_{\mathbf N}(\sigma,\xi)
	\geq\;&
	e^{ o(\mathbf{N})} \exp {\left( \Pi_q(Y) \lambda(b)\right)} (1-b)^{2 \Pi_q(Y)} =  \exp \left( {o(\mathbf{N}) - \Pi_q(Y) \ln\left(\frac{(1+b)^2}{1+b^2}\right)  }\right).
\end{align}
From our choice of $Y$ and Lemma \ref{produc of poly}, it follows that
\[ 
\Pi_q(Y)=\mathcal{Z}(2) (1+o(1)) \frac{q^Y}{Y} \leq  \left(a\mathcal Z(2)+o(1)\right)\mathbf N.
\]
Hence, using our choice of $a$, \eqref{phi> exp(-piQY)} implies that
\begin{align*}
	\Phi_{\mathbf N}(\sigma,\xi)
	&\geq
	\exp\left( 
	-a\mathcal Z(2) \mathbf N  \ln\left(\frac{(1+b)^2}{1+b^2}\right) 
	+o(\mathbf N) \right) \\
	& = \exp\left( 
	-\frac{a\alpha_q(b)}{2}
	\mathbf N\ln q
	+o(\mathbf N) \right) = q^{-\frac{\mathbf N}{2}
		(1-\xi\alpha_q(b))
		+o(\mathbf N)},
\end{align*}
This completes the proof.

\section{Acknowledgments}

The author is supported by the Institute Fellowship from IIT Kharagpur, India.
The author is grateful to Dr.~Kamalakshya Mahatab for his valuable guidance and support,
and also acknowledges financial support from the ARG-MATRICS Grant
(Grant No.~ANRF\slash ARGM\slash 2025\slash 002540\slash MTR)
awarded to Dr.~Kamalakshya Mahatab.

\end{document}